\documentclass[smallextended]{svjour3}
\usepackage{mathptmx}
\usepackage{amsmath}
\usepackage{amssymb}
\usepackage[utf8]{inputenc}
\usepackage{newunicodechar}
\newunicodechar{ℝ}{\mathbb{R}}

\begin{document}
\title{Weak Markov Operators}
\author{H\={u}lya Duru \and Serkan Ilter}


\institute{H. Duru \at
              Istanbul University, Faculty of Science, Mathematics Department,
Vezneciler-Istanbul, 34134, Turkey \\
              Tel.: +90-212-4555700 (15315)\\
              \email{hduru@istanbul.edu.tr}           %  \\
           \and
           S. Ilter \at
              Istanbul University, Faculty of Science, Mathematics Department,
Vezneciler-Istanbul, 34134, Turkey \\
              \email{ilters@istanbul.edu.tr}
}

\date{Received: date / Accepted: date}
%\The correct dates will be entered by the editor
\maketitle

\begin{abstract}
Let $A$ and $B$ be $f$-algebras with unit elements $e_{A}$ and $e_{B}$
respectively. A positive operator $T$ from $A$ to $B$ satisfying $T\left(
e_{A}\right) =e_{B}$ is called a Markov operator. In this definition we
replace unit elements with weak order units and, in this case, call $T$ to
be a weak Markov operator. In this paper, we characterize extreme points of
the weak Markov operators.
\end{abstract}


\section{INTRODUCTION}

Let $A$ and $B$ be $f$-algebras with point separating order duals and unit
elements $e_{A}$ and $e_{B}$ respectively. A positive linear operator $T$
from $A$ to $B$ satisfying $T\left( e_{A}\right) =e_{B}$ is called a Markov
operator. In this definition we replace unit elements with weak order units
and, in this case, call $T$ to be a weak Markov operator.The set of all weak
Markov operators is convex. In this paper, we characterize extreme points of
the weak Markov operators. In this regard, first we give an alternate and
quick proof of \cite{Ref6}\textrm{\ }for $f$-algebras instead of
order complete real vector spaces (Proposition \ref{p3}). Then we give an alternate
proof of Theorem $5.7$ in \cite{Ref5}\textrm{\ }(Theorem \ref{p3}). In
addition we present another necessary and sufficient condition to this
Theorem. Then we show that a weak Markov operator is a lattice homomorphism
if and only if it is an extreme point in the collection of all weak Markov
operators from $A$ into $B$ provided $B$ is order complete.

\section{PRELIMINARIES}

For unexplained terminology and the basic results on vector lattices and
semiprime $f$-algebras we refer to \cite{Ref1,Ref10,Ref13}. Let us recall some definitions.

\begin{definition}\label{d1}
The real algebra $A$ is called a Riesz algebra or lattice-ordered algebra if
$A$ is a Riesz space such that $ab\in A$ whenever $a,b$ are positive
elements in $A.$ The Riesz algebra is called an $f$-algebra if $A$ satisfies
the condition that
\begin{equation*}
a\wedge b=0\text{ implies }ac\wedge b=ca\wedge b=0\text{ for all }0\leq c\in
A\text{.}
\end{equation*}
\end{definition}

In an Archimedean $f$-algebra $A$, all nilpotent elements have index $2$.
Throughout this paper \thinspace $A$ will show an Archimedean semiprime $f$%
-algebra with point separating order dual. By definition, if zero is the
unique nilpotent element of $A$, that is, $a^{2}=0$ implies $a=0$, $A$ is
called semiprime $f$-algebra. It is well known that every $f$-algebra with
unit element is semiprime.

\begin{definition}\label{d2}
Let $A$ be an Archimedean Riesz space and let $\widehat{A}$ be the Dedekind
completion of $A$. The closure of $A$, $\overline{A}$, in $\widehat{A}$ with
respect to the relatively uniform topology \cite{Ref8}\textrm{\ }is
so called that relatively uniformly completion of $A$ \cite{Ref9}%
\textrm{. }
\end{definition}

If $A$ is an semiprime $f$-algebra then the multiplication in $A$ can be
extended in a unique way into a lattice ordered algebra multiplication on $%
\overline{A}$ such that $A$ becomes a sub-algebra of $\overline{A}$ and $%
\overline{A}$ is an relatively uniformly complete semiprime $f$-algebra. We
also recall that $\overline{A}$ satisfies the Stone condition (that is, $%
x\wedge nI^{\ast }\in \overline{A}$, for all $x\in \overline{A}$, where $I$
denotes the identity on $A$ of $OrthA$) due to Theorem $2.5$ in \cite{Ref4}. 
For $a\in A,$ the mapping $\pi _{a}:A\rightarrow OrthA,$ defined
by $\pi _{a}\left( b\right) =a.b$ is an orthomorphism on $A.$ Since $A$ is a
Archimedean semiprime $f$-algebra,$,$ the mapping $\pi :A\rightarrow OrthA,$
defined by $\pi \left( a\right) =\pi _{a}$ is an injective $f$-algebra
homomorphism. Hence we shall identify $A$ with $\pi \left( A\right) .$The
ideal center $Z\left( A\right) $of $A$ is defined as the order ideal in $%
OrthA$ generated by the identity mapping $I$ which is a unital f-algebra.

\begin{proposition}\label{p1}
\textbf{\ }Let $A$ and $B$ be semiprime $f$-algebras and $T:A\rightarrow B$
an order bounded linear operator. If $T\left( A\cap Z\left( A\right) \right)
=\{0\}$ then $T\left( A\right) =\{0\}$.
\end{proposition}

\begin{proof}
$T$ can be regarded as an element of the collection of all order bounded
linear operators from $A$ to $\widehat{B}$. Therefore there exist two order
bounded positive operators $T_{1}$ and $T_{2}$ such that $T=T_{1}-T_{2}$.
Thus we can assume that $T$ is positive. Since $\widehat{B}$ is relatively
uniformly complete, by Theorem $3.3$ in \cite{Ref11} there exists a
unique positive relatively uniformly continuos extension $\overline{T}:%
\overline{A}\rightarrow \widehat{B}$ of $T$ to the relatively uniformly
completion $\overline{A}$ of $A$, defined by,%
\begin{equation*}
\overline{T}\left( x\right) =\sup \left\{ Ta:0\leq a\leq x\right\}
\end{equation*}%
for $x\in \overline{A}$. Let $0\leq x\in \overline{A}\cap Z\left( \overline{A%
}\right) $ and $a\in A\cap \left[ 0,x\right] $. Then $a\leq \lambda I_{A}$
holds for some positive real number $\lambda $. By the hypothesis, $Ta=0$.
From here we conclude that $\overline{T}\left( x\right) =0$ for all $x\in
\overline{A}\cap Z\left( \overline{A}\right) $. Let $0\leq a\in A$. Then $%
\left( a\wedge n\overline{I}\right) _{n}$ is a sequencein $\overline{A}\cap
Z\left( \overline{A}\right) $ and it is converging to $a$ by Proposition $%
2.1\left( i\right) $ in \cite{Ref5}. It follows from the relatively
uniformly continuity of $\overline{T}$ that $0=\overline{T}\left( a\wedge
nI_{A}\right) \rightarrow \overline{T}\left( a\right) =T\left( a\right) $
(relatively uniformly). Thus $T\left( a\right) =0$ for all $a\in A$.
\end{proof}

\begin{proposition}\label{p2}
Let $A$ and $B$ be $f$-algebras with unit elements $e_{A}$ and $e_{B}$
respectively and $T:A\rightarrow B$ an order bounded linear operator
satisfying $T\left( e_{A}\right) =e_{B}$. If $\mid Ta\mid \leq e_{B}$
whenever $\mid a\mid \leq e_{A}$, then $T$ is positive.
\end{proposition}

\begin{proof}
$T$ can be regarded as an element of the collection of all order bounded
linear operators from $A$ to $\widehat{B}$. As $\widehat{B}$ is Dedekind
complete, $\mid Ta\mid \leq \mid T\mid \left( a\right) =\sup \left\{ \mid
Tb\mid :\mid b\mid \leq a,b\in A\right\} $ holds for all $0\leq a\in A$. Let
$a\in A$ be an element such that $\mid a\mid \leq e_{A}$. Then by the
hypothesis, $\mid Ta\mid \leq e_{B}$. This shows that $\mid Te_{A}\mid
=e_{B}\leq \mid T\mid \left( e_{A}\right) \leq e_{B}$. Thus $\mid T\mid
\left( e_{A}\right) =e_{B}$. Let $a$ $\in A\cap Z\left( A\right) $. Then
there exists a positive real number $\lambda $ such that $\mid a\mid \leq
\lambda e_{A}$. From here we derive that%
\begin{equation*}
\left( \mid T\mid -T\right) \left( \mid a\mid \right) \leq \lambda \left(
\mid T\mid -T\right) \left( e_{A}\right) =0
\end{equation*}%
Thus the order bounded operator $\mid T\mid -T$ vanishes on $A\cap Z\left(
A\right) $. By proposition $1$, we conclude that $\mid T\mid =T$ on $A$, so $%
T$ is positive.
\end{proof}

The following proposition was proved in \cite{Ref6}\textrm{\ }for an
order complete real vector space. In the following proposition we will give
an alternate and quick proof for Archimedean semiprime $f$-algebras.

\begin{proposition}\label{p3}
Let $A$ be a semiprime $f$-algebra. A positive element $a_{0}$ in $A$ is a
weak order unit if and only if \ $Inf\left\{ \mid a-\lambda a_{0}\mid
:\lambda \in
%TCIMACRO{\U{211d} }%
%BeginExpansion
\mathbb{R}
%EndExpansion
\right\} =0$ for all $a\in A$.
\end{proposition}

\begin{proof}
Let $0<a_{0}$ $\in $ $A$ be a weak order unit. Assume that, on the contrary,
there exist $a\in A$ and $b$ $\in $ $\widehat{A}$ satisfying $0<b\leq \mid
a-\lambda a_{0}\mid $for all $\lambda \in
%TCIMACRO{\U{211d} }%
%BeginExpansion
\mathbb{R}
%EndExpansion
$. Then
\begin{equation*}
\lambda ^{2}a_{0}^{2}-2\lambda aa_{0}+a^{2}-b^{2}\geq 0\text{ for all }%
\lambda \in
%TCIMACRO{\U{211d} }%
%BeginExpansion
\mathbb{R}
%EndExpansion
\end{equation*}%
Using Proposition $3.3$ in \cite{Ref3}, we get $\left( a_{0}b\right)
^{2}=0$. Since $\widehat{A}$ is semiprime, the last equality implies that%
\begin{equation*}
0=a_{0}b=\sup \left\{ aa_{0}\in A:a\leq b,\text{ }a\in A\right\} \text{.}
\end{equation*}%
As $a_{0}$ is weak order unit, $aa_{0}=0$ implies that $a=0$ for all $a\in A$
satisfying $a\leq b$. Thus $b=0$. Conversely, suppose that there exists $%
0\neq b\in A$ such that $ba_{0}=$ $0$. Then since $0<\frac{\mid b\mid }{2}%
\leq \mid b-\lambda a_{0}\mid $ for all $\lambda \in
%TCIMACRO{\U{211d} }%
%BeginExpansion
\mathbb{R}
%EndExpansion
$, we get $Inf\left\{ \mid b-\lambda a_{0}\mid :\lambda \in
%TCIMACRO{\U{211d} }%
%BeginExpansion
\mathbb{R}
%EndExpansion
\right\} \neq 0$.
\end{proof}

A positive linear operator $T$ between two unital $f$-algebras $A$ and $B$
is \ said to be a Markov operator for which $T\left( e_{A}\right) =e_{B}$
where $e_{A}$, $e_{B}$ are the unit elements of $A$ and $B$ respectively. In
both \cite[Theorem~5.7]{Ref5}\textrm{, }it is proved that an
operator is an extreme point of Markov operators if and only if it is a
lattice homomorphism. In the following theorem we shall give an alternate
proof of this theorem and another sufficient and necessary condition for the
extreme point of a Markov operator.

\begin{theorem}\label{t1}
Let $A$ and $B$ be $f$-algebras with unit elements $e_{A}$ and $e_{B}$
respectively and $T:A\rightarrow B$ an order bounded linear operator
satisfying $T\left( e_{A}\right) =e_{B}$ and $\mid Ta\mid \leq e_{B}$
whenever $\mid a\mid \leq e_{A}$. Then the following are equivalent;\newline
$\left( i\right) $ $T$ is an extreme point of Markov operators.\newline
$\left( ii\right) $ $Inf\left\{ T\mid a-\lambda e_{A}\mid :\lambda \in
%TCIMACRO{\U{211d} }%
%BeginExpansion
\mathbb{R}
%EndExpansion
\right\} =0$ for all $a\in A$.\newline
$\left( iii\right) $ $T$ is a lattice homomorphism.
\end{theorem}

\begin{proof}
$\left( i\right) \Rightarrow \left( ii\right) :$ By Proposition  \ref{p2}, $T$ is
positive. Suppose, on the contrary, that $Inf\left\{ T\mid a_{0}-\lambda
e_{A}\mid :\lambda \in
%TCIMACRO{\U{211d} }%
%BeginExpansion
\mathbb{R}
%EndExpansion
\right\} \neq 0$ for some $0\leq a_{0}\in A$. Then there exists $b\in
\widehat{B}$ such that $0<b\leq T\mid a_{0}-\lambda e_{A}\mid $ for all $%
\lambda \in
%TCIMACRO{\U{211d} }%
%BeginExpansion
\mathbb{R}
%EndExpansion
$. Taking square from both side and applying the Schwarz inequality \cite[Corollary~3.5] {Ref3}, one can get easily%
\begin{equation*}
\lambda ^{2}e_{B}-2\lambda Ta_{0}+T\left( a_{0}^{2}\right) -b^{2}\geq 0\text{
for all }\lambda \in
%TCIMACRO{\U{211d} }%
%BeginExpansion
\mathbb{R}
%EndExpansion
\text{.}
\end{equation*}%
Taking into account Proposition $3.3$ in \cite{Ref3}, we derive that%

\begin{equation}
\left( Ta_{0}\right) ^{2}\leq T\left( a_{0}^{2}\right) -b^{2} \label{eq1}.
\end{equation}%

Let, $a_{n}=\frac{a_{0}}{n}\wedge e_{A}$ $\left( n=1,2,...\right) $. Then $%
0<a_{n}$, as $Inf\left\{ T\mid a_{0}-\lambda e_{A}\mid :\lambda \in
%TCIMACRO{\U{211d} }%
%BeginExpansion
\mathbb{R}
%EndExpansion
\right\} \neq 0$, for each $n$. Define,
\begin{equation}
T_{a_{n}}\left( b\right) =T\left( a_{n}b\right) -T\left( a_{n}\right)
T\left( b\right) \label{eq2}
\end{equation}%
for $b\in A$. Clearly the operators $T\mp T_{a_{n}}$ are positive and $%
T_{a_{n}}\left( e_{A}\right) =0$. Hence $T\mp T_{a_{n}}$ are Markov
operators. From here we conclude that $T_{a_{n}}=0$, as $T$ is the extreme
point of Markov operators. Hence the equality (\ref{eq2}) implies
that
\begin{equation*}
nT\left( a_{n}b\right) =nT\left( a_{n}\right) T\left( b\right)
\end{equation*}%
and by setting $a_{n}$, we have $T(\left( a_{0}\wedge ne_{A}\right)
a_{0})=T\left( a_{0}\wedge ne_{A}\right) T\left( a_{0}\right) $. Applying
Proposition $2.1\left( i\right) $ in \cite{Ref5}\textrm{\ }and using
the positivity of $T$, one may get,
\begin{equation*}
T\left( a_{0}^{2}\right) =\left( T\left( a_{0}\right) \right) ^{2}\text{.}
\end{equation*}
Thus the ineaquality (\ref{eq1}) implies that $b=0$, as $\overset{%
\wedge }{B}$ is semiprime whenever $B$ is semiprime.\newline
$\left( ii\right) \Rightarrow \left( iii\right) :$ Let $a\in A$ and $\lambda
\in
%TCIMACRO{\U{211d} }%
%BeginExpansion
\mathbb{R}
%EndExpansion
$. Then%
\begin{equation*}
T\mid a\mid \leq T\mid a-\lambda e_{A}\mid +T\mid \lambda e_{A}\mid =T\mid
a-\lambda e_{A}\mid +\mid T(\lambda e_{A})\mid \leq 2T\mid a-\lambda
e_{A}\mid +\mid Ta\mid
\end{equation*}%
it follows from $\left( ii\right) $ that $T\mid a\mid \leq \mid Ta\mid $ and
since $T$ is positive, $T\mid a\mid =\mid Ta\mid $ which implies that $T$ is
a lattice homomorphism.\newline
$\left( iii\right) \Rightarrow \left( i\right) :$ By Proposition \ref{p3}, we get
that $Inf\left\{ \mid Ta-\lambda e_{B}\mid :\lambda \in
%TCIMACRO{\U{211d} }%
%BeginExpansion
\mathbb{R}
%EndExpansion
\right\} =0$ for all $a\in A$, as $e_{B}$ is a weak order unit in $B$. To
prove that $T$ is an extreme point of Markov operators it is enough to show
that for any Markov operator $S$ and $0<\alpha \in
%TCIMACRO{\U{211d} }%
%BeginExpansion
\mathbb{R}
%EndExpansion
$ satisfying $\alpha T-S\geq 0$ implies that $T=S$. Let $a\in A$ and $%
\lambda \in
%TCIMACRO{\U{211d} }%
%BeginExpansion
\mathbb{R}
%EndExpansion
$. It follows from%
\begin{equation*}
\mid Sa-S\left( \lambda e_{A}\right) \mid \leq S\mid a-\lambda e_{A}\mid
\leq \alpha T\mid a-\lambda e_{A}\mid =\alpha \mid Ta-\lambda e_{B}\mid
\end{equation*}%
that%
\begin{equation*}
\mid Ta-Sa\mid \leq \mid Ta-\lambda e_{B}\mid +\mid Sa-\lambda e_{B}\mid
\leq \left( 1+\alpha \right) \mid Ta-\lambda e_{B}\mid
\end{equation*}%
Hence $Ta=Sa$ for $a\in A$.
\end{proof}

\begin{definition}\label{d3}
Let $A$ and $B$ be $f$-algebras with point separating order duals and weak
order units $e_{A}$ and $e_{B}$ respectively. In this case, we call a
positive linear operator $T$ from $A$ to $B$ satisfying $T\left(
e_{A}\right) =e_{B}$ to be a weak Markov operator $\left( \text{briefly WMO}%
\right) $.
\end{definition}

Now we remark that in the last step $\left( iii\right) \Rightarrow \left(
i\right) $ of the above proof we proved the following Corollary as well;

\begin{corollary}\label{c1}
Let $A$ and $B$ be a semiprime $f$-algebras with weak order units $e_{A}$
and $e_{B}$ respectively and let $T:A\rightarrow B$ be a weak Markov
operator. If $T$ is a lattice homomorphism, then it is an extreme point in
the set of all weak Markov operators.
\end{corollary}

In \cite[Lemma~4.2] {Ref2}, it was proved that a positive
contractive Stone operator between two Stone $f$-algebras is a lattice
homomorphism. At this point we remark that this result is true for all
positive linear operator between any semiprime $f$-algebras. For the
completeness, we repeat this proof in the following Proposition.

\begin{proposition}\label{p4}
Let $A$ and $B$ be semiprime $f$- algebras and $T:A\rightarrow B$ a positive
linear operator. If there exists $0\leq a_{0}\in A$ and $0\leq b_{0}\in B$
such that $T\left( a\wedge a_{0}\right) =T\left( a\right) \wedge b_{0}$, for
all $a\in A$, then $T$ is a lattice homomorphism.
\end{proposition}

\begin{proof}
We remark that, for all elements $a$ in $A$ or $B$, the following equations
hold;%
\begin{eqnarray*}
\left( a\wedge a_{0}\right) ^{+} &=&a^{+}\wedge a_{0} \\
\left( a\wedge a_{0}\right) ^{-} &=&a^{-}
\end{eqnarray*}%
Since $T$ is positive, in order to prove that $T$ is a lattice homomorphism
it is enough to show that $T(a^{+})\leq (Ta)^{+}$ for all $a\in A$. By above
remark,%
\begin{eqnarray*}
Ta &=&T\left( a^{+}-a^{-}-(a\wedge \frac{1}{n}a_{0})^{+}+(a\wedge \frac{1}{n}%
a_{0})^{-}\right) +T(a\wedge \frac{1}{n}a_{0}) \\
&=&T\left( a^{+}-(a\wedge \frac{1}{n}a_{0})^{+}\right) +(T(a\wedge \frac{1}{n%
}a_{0}))
\end{eqnarray*}%
and
\begin{eqnarray*}
Ta &=&\left( Ta\right) ^{+}-\left( Ta\right) ^{-}-(T(a\wedge \frac{1}{n}%
a_{0}))^{+}+(T(a\wedge \frac{1}{n}a_{0}))^{-}+T(a\wedge \frac{1}{n}a_{0}) \\
&=&\left( Ta\right) ^{+}-(T(a\wedge \frac{1}{n}a_{0}))^{+}+(T(a\wedge \frac{1%
}{n}a_{0}))
\end{eqnarray*}%
combining these results, we get%
\begin{equation*}
T\left( a^{+}\right) -\left( Ta\right) ^{+}=^{-}T\left( a^{+}\wedge \frac{1}{%
n}a_{0}\right) -\left( Ta\right) ^{+}\wedge \frac{1}{n}b_{0}\leq \frac{1}{n}%
T\left( a_{0}\right)
\end{equation*}%
Passing limit as $n\rightarrow \infty $, we have the desired result.
\end{proof}

\begin{corollary}\label{c2}
\label{7}Let $A$ and $B$ be semiprime $f$- algebras with weak order units $%
e_{A}$ and $e_{B}$ respectively and $T:A\rightarrow B$ a weak Markov
operator. If there exists $0\leq a_{0}\in A$ and $0\leq b_{0}\in B$ such
that $T\left( a\wedge a_{0}\right) =T\left( a\right) \wedge b_{0}$, for all $%
a\in A$, then $T$ is an extreme point of weak Markov operators.
\end{corollary}

\begin{proof}
By Proposition \ref{p4}, $T$ is a lattice homomorphism and by Corollary \ref{c1}, $T$
is an extreme point of weak Markov operators.
\end{proof}

\begin{theorem}\label{t2}
Let $A$ and $B$ be a semiprime $f$-algebras with weak order units $e_{A}$
and $e_{B}$ respectively and $T:A\rightarrow B$ a weak Markov operator. If $%
B $ is order complete, then $T$ is an extreme point of the set of all weak
Markov operators if and only if $T$ is a lattice homomorphism.
\end{theorem}

\begin{proof}
First we remark that any weak Markov operator is an extension of the
operator $f:M\rightarrow B$, defined by $f\left( \lambda e_{A}\right)
=\lambda e_{B}$, where $M=\left\{ \lambda e_{A}:\lambda \in
%TCIMACRO{\U{211d} }%
%BeginExpansion
\mathbb{R}
%EndExpansion
\right\} $. It is not difficult to see that the extreme point of the
collection of all weak Markov operators is the same of the extreme point of
all positive extensions of $f$ to the whole $A$. Let $T~$\ be an extreme
point in the set of all weak Markov operators. Then $T$ is an extreme point
of the set of all positive extensions of $f$ to $A$. Since $f$ is a lattice
homomorphism, by Theorem $2\left( a\right) $ in \cite{Ref7},\textrm{\
}we derive that $T$ is a lattice homomorphism. Conversely assume that $T$ is
a lattice homomorphism. Taking into account Proposition \ref{p3} and Theorem $%
2\left( b\right) $ in \cite{Ref7}\textrm{,} we conclude that $T$ is
an extreme point of the weak Markov operators.
\end{proof}


\begin{thebibliography}{30}
\bibitem{Ref1} Aliprantis, C.D., Burkinshaw, O.: Positive Operators, Academic Press,
Orlando (1985)

\bibitem{Ref2} Ben Amor, M.A., Boulabiar, K., El Adeb, C.: Extreme
contractive operators on Stone $f$-algebras, Indagationes Mathematicae
\textbf{25}, 93-103 (2014)

\bibitem{Ref3} Huijsmans, C.B., de Pagter, B Averaging operators and positive
contractive projections, Journal of Mathematical Analysis and applications
\textbf{113}, 163-184 (1986)

\bibitem{Ref4} Huijsmans, C.B., de Pagter, B.: The order bidual of
lattice-ordered algebras, J.Funct.Analysis \textbf{59}, 41-64 (1984)

\bibitem{Ref5} Huijsmans, C.B., de Pagter, B.: Subalgebras and Riesz
subspaces of an f-algebra, Proc.Lond.Math.Soc. \textbf{48}, 161-174 (1984)

\bibitem{Ref6} Lipecki, Z., Thomsen, W.: Extension of positive operators and exterma
points IV, Colloq.Math.  \textbf{46}, 269-273 (1982)

\bibitem{Ref7} Lipecki, Z., Extensions of positive operators and extreme points. II,
Ibidem \textbf{42}, 285-289 (1979)

\bibitem{Ref8} Luxemburg,W.A.J., Zaanen A.C. : Riesz Spaces I, North Holland,
Amsterdam (1971)

\bibitem{Ref9} Quinn, J.: Intermediate Riesz spaces, Pac.J.Math.\textbf{56}, 255-263 (1975)

\bibitem{Ref10} Schaefer H.H.: Banach Lattices and Positive Operators, Springer,
Berlin (1974)

\bibitem{Ref11} Triki, A.: Algebra homomorphisms in complex almost f-algebras, Comment
Math.Univ.Carol.\textbf{43}, 23-31 (2002)

\bibitem{Ref12} Van Putten, B., Disjunctive linear operators and partial
multiplication in Riesz spaces, Ph.D Thesis, Wageningen (1980)

\bibitem{Ref13} Zaanen A.C.: Riesz Spaces II, North- Holland, Amsterdam (1983)

\end{thebibliography}
\end{document}